\documentclass[11pt]{article}
\usepackage[dvips]{graphics}
\usepackage{epsfig,rotating}
\usepackage{amssymb,amsmath,amsthm}
\usepackage{latexsym,natbib}
\usepackage[usenames]{color}

\newcommand{\be}{\begin{equation}}
\newcommand{\ee}{\end{equation}}
\newcommand{\beaa}{\begin{eqnarray*}}
\newcommand{\eeaa}{\end{eqnarray*}}
\newcommand{\bea}{\begin{eqnarray}}
\newcommand{\eea}{\end{eqnarray}}
\newcommand{\lbl}{\label}

\newcommand{\bSigma}{\mathbf{\Sigma}}
\newcommand{\hSigma}{\mathbf{\hat\Sigma}}

\newtheorem{theorem}{ \noindent T{\footnotesize HEOREM}}

\newtheorem{lemma}{ \noindent L{\footnotesize EMMA}}
\newtheorem{coro}{ \noindent C{\footnotesize OROLLARY}}

\begin{document}

\title{Joint limiting laws for high-dimensional independence tests\thanks{Both authors contributed equally to this paper. 
The authors thank Stephane Bonhomme, Zongming Ma, Richard Samworth, Jeff Yao, Cun-Hui Zhang, and seminar/conference participants at University of Chicago, Peking University, and 2015 CMStatistics for their helpful comments and suggestions. This work was initiated when Danning Li was a postdoc at University of Cambridge. Danning Li's research was supported by the Engineering and Physical Sciences Research Council Early Career Fellowship EP/J017213/1. Lingzhou Xue's research is supported by the American Mathematical Society Simons Travel Grant and the National Science Foundation grant DMS-1505256.}}
\author{Danning Li and Lingzhou Xue\\
Jilin University and Pennsylvania State University
}

\date{First Version: August 2015; This Version: December 2015}
\maketitle


\begin{abstract}


Testing independence is of significant interest in many important areas of large-scale inference. Using extreme-value form statistics to test against sparse alternatives and using quadratic form statistics to test against dense alternatives are two important testing procedures for high-dimensional independence. However, quadratic form statistics suffer from low power against sparse alternatives, and extreme-value form statistics suffer from low power against dense alternatives with small disturbances and may have size distortions due to its slow convergence. For real-world applications, it is important to derive powerful testing procedures against more general alternatives. Based on intermediate limiting distributions, we derive (model-free) joint limiting laws of extreme-value form and quadratic form statistics, and surprisingly, we prove that they are asymptotically independent. Given such asymptotic independencies, we propose (model-free) testing procedures to boost the power against general alternatives and also retain the correct asymptotic size. Under the high-dimensional setting, we derive the closed-form limiting null distributions, and obtain their explicit rates of uniform convergence. We prove their consistent statistical powers against general alternatives. We demonstrate the performance of our proposed test statistics in simulation studies. Our work provides very helpful insights to high-dimensional independence tests, and fills an important gap.
 \end{abstract}



\section{Introduction}

Inference of high-dimensional data is now becoming increasingly important in theoretical and applied research of statistics, econometrics, machine learning, and signal processing. Testing the covariance structure such as independence and bandedness has received growing attention in various areas of high-dimensional inference. High-dimensional independence plays a critical role in many important methods, which usually assume the normality of observations. \cite{benjamini-hochberg-1995} and \cite{storey-2002} effectively controlled the false discovery rate for independent test statistics in large-scale multiple testing. \cite{bickel-2004} and \cite{fan-fan-2008} advocated the independence rule in Fisher's linear discriminant analysis for high-dimensional classification. \cite{baltagi-etal-2012} and \cite{fan-etal-2015} also emphasized the importance to test cross-sectional independence.

Let {$\mathbf{x}_1, \ldots, \mathbf{x}_n $} be $n$ independent observations of the $p$-dimensional real-valued random vector ${\mathbf{x}=(x_1, \ldots, x_p)'}$. The focus of this paper is on addressing the following hypothesis testing problem of significant interest:
$$
\mathbf{H}_0: \ x_1, \ldots, x_p \ \text{are mutually independent,}
$$
under the ultra high-dimensional setting that dimension $p$ can be on a nearly exponential rate of sample size $n$. It is equivalent to test the diagonality of the covariance matrix $\mathbf{\Sigma}$ when we assume the normality of observations. Without loss of generality, we assume that $x_1, \ldots, x_p$ have zero means and unit variances.

The independence testing problem was extensively studied in the classical setting that $p$ is fixed and $n$ diverges to infinity. Many traditional methods include the likelihood ratio test (LRT), Roy's largest eigenvalue test, John's test, and Nagao's test. See \cite{muirhead1982} and \cite{anderson1984} for more details. Driven by rapid developments of data collection techniques, modern scientific and engineering applications pay more attention to the emerging setting that both $n$ and $p$ diverge to infinity. In this new setting, conventional testing procedures often perform quite poorly, and some of them may even be not well-defined. In the past fifteen years, statisticians and econometricians have made important advances, and proposed new testing procedures, which can be categorized into two different asymptotic regimes. In the first decade, investigations were very active in the high-dimensional setting that $p$ and $n$ are comparable (i.e., $p/n\rightarrow \rho \in(0,\infty)$). \cite{johnstone2001} revisited Roy's largest eigenvalue test and proved the Tracy--Widom law for its limiting distribution under null hypothesis. \cite{ledoit-wolf-2002} also extended John's test and Nagao's test to such setting. \cite{bai2009} and \cite{jiang2013} proposed the corrected LRTs and proved their asymptotic normality. Recently, the central focus in on the ultra high-dimensional setting that $p$ is much larger than $n$ (i.e., $p/n\rightarrow \infty$). Both the LRT and the corrected LRTs are not well defined now. To tackle this challenging asymptotic regime, two alternative hypotheses are mainly considered in the current literature:

\begin{itemize}
   \item Testing high-dimensional independence against dense alternatives, where $\mathbf{H}_0$ is violated by many small off-diagonal elements in $\mathbf{\Sigma}$. Under certain moment conditions, the seminal paper by \cite{chen-etal-2010} proposed two new testing procedures without explicitly specifying an relationship between $p$ and $n$. The statistics in \cite{chen-etal-2010} are modified from the following quadratic form statistic:
       \be
       {S_n=\sum_{1\leq i<j\leq p}\hat{\sigma}_{ij}^2},
       \ee
       the sum of squares of the off-diagonal entries of $\hat{\Sigma}$, where $\hSigma=(\hat{\sigma}_{ij})_{p\times p}$ is the sample covariance matrix. As shown in Lemma \ref{lawSn}, we can obtain the central limit theorem that $
        {b_n(nS_n-a_n)\to N(0,1)}
        $, where $a_n$ and $b_n$ are specified in Section 2.

   \item Testing high-dimensional independence against sparse alternatives, where $\mathbf{H}_0$ is violated by only a few large elements in $\mathbf{\Sigma}$. Under similar moment conditions, \cite{caijiang-2011} instead considered the following extreme-value form statistic:
       \be\label{mount}
       {L_n=\max_{1\leq i<j\leq p}|\hat{\sigma}_{ij}|},
       \ee
       the largest magnitude of the off-diagonal entries of $\hSigma$, which is also known as coherence of the random matrix in signal processing. \cite{caijiang-2011} extended \cite{jiang2004} to obtain that $P(nL_n^2\leq y+4 \log p-\log (\log p))\to e^{{-1}/{(\sqrt{8\pi})}e^{-{y}/2}}$, whose limiting distribution is an extreme distribution of type I (i.e., a Gumbel distribution).
\end{itemize}
Without requiring any model assumptions, we may also obtain the rank-based test statistics for testing high-dimensional independence against dense alternatives or against sparse alternatives. Remark that the second-order spectral test statistic in \cite{bao-etal-2015} and the nonparametric test statistic in \cite{zhou-2007} and \cite{han2014} correspond to the rank-based alternatives to $S_n$ and $L_n$ respectively. For notational convenience, we denote them by $T_n$ and $M_n$. Note that $T_n$ and $M_n$ obtain the similar limiting distributions as their counterparts.

However, it is well-known that quadratic form statistics suffer from low power against sparse alternatives, and extreme-value form statistics suffer from low power against dense alternatives with small disturbances. We illustrate such undesirable numerical performances for $S_n$ and $L_n$, $T_n$ and $M_n$ in our simulation studies in Section 5. In real-world applications, it would be very important and also appealing to derive a powerful testing procedure against more general alternatives if possible. But it even brings us more difficulty that extreme-value form statistics often suffer from size distortions due to their slow convergence, as pointed out in \cite{hall-1979}.

In this work, we shall address these challenging issues from a completely different perspective to existing methods. We first study the joint limiting law of $S_n$ and $L_n$ in Theorems 1--2 and the model-free joint limiting law of $T_n$ and $M_n$ in Theorems 3--4. Surprisingly, we show that $S_n$ and $L_n$ are indeed asymptotically independent, and so are $T_n$ and $M_n$. To the best of our knowledge, our work is the first in the statistics and econometrics literatures to derive the asymptotic independence of the quadratic form statistics and the extreme-value form statistics in high-dimensional inference. With the aid of intermediate limiting distributions \citep{liulinshao-2008}, both joint limiting laws are derived under uniform convergence with closed-form limiting null distributions, and their explicit rates of uniform convergence are also obtained. Given asymptotic independencies, we propose (model-free) high-dimensional independence testing procedures, i.e., $TS^1_n$ based on $S_n$ and $L_n$ and $TS^2_n$ based on $T_n$ and $M_n$, to boost the power against general alternatives and also retain the correct asymptotic size. Similar results can be obtained for testing bandedness of the covariance matrix. We derive their asymptotic sizes from the convolution of their joint limiting laws, and prove the consistent power of $TS^1_n$ against more general alternatives. Our work provides very helpful insights to high-dimensional independence tests, and fills an important gap in the current literature.


The rest of this paper is organized as follows. Sections 2 and 3 derive the joint limiting laws with explicit limiting distributions and rates of uniform convergence. Section 4 introduces our proposed test statistics, and establishes the theoretical fondations for their sizes and powers against more general alternatives. Section 5 demonstrates the numerical performance of proposed testing procedures. Section 6 includes key technical proofs, and Section 7 has a few concluding remarks. More technical details are rendered in the technical report version of this paper \citep{lixue-2015}.

\section{Joint Limiting Law}
\lbl{jointlaw1}

Section 2 proves the joint limiting law of two test statistics $S_n$ and $L_n$ under the high-dimensional setting that $p$ is either comparable to $n$ or much larger than $n$. Before proceeding, we first introduce some necessary notation. Let $\mathcal{X}=(\mathbf{x}_1,\cdots,\mathbf{x}_n)'=(x_{ij})_{n\times p}$ be the observed $n\times p$ random matrix with $p\ge n$. We assume that $\{x_{ij}:\, 1 \leq i\leq n, 1\le j\le p\}$, all entries of $\mathcal{X}$, are $n\times p$ independently and identically distributed random variables with zero mean and unit variance, i.e., $Ex_{ij}=0$ and $ E(x_{ij}^2)=1$. Then, each entry of the sample covariance matrix $\hSigma=(\hat{\sigma}_{ij})_{p\times p}$ has the form of
\bea\lbl{corr}
\hat{\sigma}_{ij}=\frac{X_{i}^{T}X_{j}}{n},\ \ 1\leq i, j\leq p,
\eea
where $X_i$ and $X_j$ denote the $i$-th column and the $j$-th column of $\mathcal{X}$ respectively. To simplify notation, we will use $c$ or $C$ to denote constants that do not depend on $n$ or $p$.

We define the marginal distribution functions of $S_n$ and $L_n$ as
$$P_{S_n}(z)=P\big(b_n(nS_n-a_n)\leq z\big),$$
where $a_n$ and $b_n$ are specified in Lemma \ref{lawSn}, and
$$P_{L_n}(y)=P\big(n L_{n}^2 - 4\log p + \log\log p \le y \big).$$
Moreover, we introduce their joint distribution function as
$$
P_{{S_n,L_n}}(z,y)=P\big(\{b_n(nS_n-a_n)\leq z\}\cap \{n L_{n}^2 - 4\log p + \log\log p \ge y\} \big).
$$
Note that $P_{{S_n,L_n}}(z,y)$ computes the probability of $\{L_{n}^2 - 4\log p + \log\log p \ge y\}$ instead of $\{L_{n}^2 - 4\log p + \log\log p \le y\}$. $P_{{S_n,L_n}}(z,y)$ is equivalent to the canonical joint distribution function, but it simplifies technical analysis and results in the explicit limiting distribution. In what follows, we first study the limiting behaviours of marginal distribution functions $P_{S_n}(z)$ and $P_{L_n}(y)$, and then explore their joint limiting law $P_{{S_n,L_n}}(z,y)$.

To derive limiting laws, we assume the following moment assumption for $\mathcal{X}$:
\begin{itemize}
  \item [(C1)] There exists some fixed constant $t_0>0$ such that $E[e^{t_0x_{1j}^{2}}]< \infty$  for all $1\leq j \leq p$.
\end{itemize}
Condition (C1) is a standard assumption for $\mathcal{X}$, and it has been used in \cite{caijiang-2011}, \cite{xuezou-2012}, among many other references.

\cite{chen-etal-2010} and \cite{caijiang-2011} have investigated the asymptotic distributions of $S_n$ and $L_n$ respectively. Especially, $S_n$ asymptotically converges to a normal distribution, while $L_n$ converges to a Gumbel distribution. Now we revisit the asymptotic distributions of $S_n$ and $L_n$, and derive their explicit uniform rates of convergence in Lemmas \ref{lawSn} and \ref{lawLn}.
Lemmas \ref{lawSn} and \ref{lawLn} extend \cite{chenshao-2007} and \cite{liulinshao-2008} from the high-dimensional setting, where $p$ is on the polynomial order of $n$, to the ultra high-dimensional setting respectively, where $p$ can be on the nearly exponential order of $n$.

\begin{lemma}\label{lawSn} Assume that $\{x_{ij};\, 1 \leq i\leq n\}$ are i.i.d. random variables with $Ex_{1j}=0$, $ E(x_{1j}^2)=1$, $Ex_{1j}^4>1$ and $E |x_{1j}|^{8}< \infty$.       Then,
\begin{eqnarray}\lbl{benz}
 \sup_z |P( b_n(nS_n-a_n)\leq z )-\Phi(z)| \leq \left\{ \begin{array}{ll}
        C \sqrt{n/p} & \mbox{if $p \gg n^{5/3}$};\\
      C     p^{-1/5}& \mbox{if $p =O(n^{5/3})$},\end{array} \right.
 \end{eqnarray}
where $\Phi(z)$ is the standard normal distribution function. Here,
$$a_n={p\choose 2}$$
and
$$ b_n= \left\{ \begin{array}{ll}
          \sqrt{n}/\left({(p-1)\sqrt{\sum_{i=1}^p(Ex_{1i}^4-1)}  }\right) & \mbox{if $p \gg n^{5/3}$}\\
        n/({p\choose 2}\sqrt{ V_n }) & \mbox{if $p =O(n^{5/3})$}\end{array} \right.
$$
where
$$
V_n=\frac{ 4(n^2-n)}{p(p-1)}+\frac{4n\sum_{1\leq i<j\leq p}(Ex_{1i}^4-1)(Ex_{1j}^4-1)}{p^2(p-1)^2}+\frac{4n\sum_{i=1}^p(Ex_{1i}^4-1)}{p^2}.
$$
\end{lemma}

\begin{lemma}\label{lawLn}
Suppose Condition (C) and also that  $\log p=o(n^{1/3})$. Then, we have
\begin{eqnarray}
\sup_{y}\Big| P\left(n L_{n}^2 - 4\log p + \log\log p \le y\right)-F(y)\Big|\leq C\sqrt{\frac{(\log p)^3}{n}}
\end{eqnarray}
where $F(y)=\exp\Big(-\frac{p^2-p}{2}P(\chi^2(1)\geq 4\log p - \log\log p + y)\Big)$.
\end{lemma}

\textbf{Remark 1}: $L_n$ is the extreme-value type statistic, and it would suffer from size distortions due to their slow convergence to the Gumbel distribution $e^{{-1}/{(\sqrt{8\pi})}e^{-{y}/2}}$. This phenomenon is similar to the slow convergence of normal extremes (Hall, 1979). To address this issue, we follow the idea of \cite{liulinshao-2008} to study the convergence of $L_n$ to its intermediate limiting distribution $F(y)$, whose final limiting distribution is $e^{{-1}/{(\sqrt{8\pi})}e^{-{y}/2}}$. With the aid of $F(y)$, Lemma \ref{lawLn} obtains the explicit rate of uniform convergence in high dimensions. Numerical performances of the intermediate limiting distribution are demonstrated in Section 5.



\smallskip
Given Lemmas \ref{lawSn} and \ref{lawLn}, we are ready to prove the surprising asymptotic independence of $S_n$ and $L_n$ in Theorem \ref{birthday}. Moreover, Theorem \ref{birthday} also establishes the explicit rate of convergence for this asymptotic result.

\begin{theorem}\lbl{birthday}  Suppose Condition (C) and also that $n,p\to \infty$ and $\log p=o(n^{1/6})$. Then,
\be
\sup_{y, z}\Big|P_{{S_n,L_n}}(z,y)-P_{S_n}(z)(1-P_{L_n}(y))\Big|\leq C\cdot \min(p^{-1/5}, \sqrt{n/p}).
\ee
\end{theorem}

It is a simple but important observation that $S_n$ is the {sum} of dependent variables $\{\hat \sigma_{ij};\, 1 \leq i<j\leq n\}$, and $L_n$ is the {maximum} of these dependent variables. We now give some intuition about the asymptotic independence in the following two remarks.

\smallskip
\textbf{Remark 2}: Theorem 1 shares the similar philosophy with some classical results. To be more specific, let $u_1,\ldots,u_n$ be independently identically distributed random variables. It is well-known that the sum $\sum_{i=1}^n u_i$ and the maximum $\max_{i=1,\ldots,n} u_i$ are asymptotically independent \citep{chow1978}. The seminal papers by \cite{hsing-1995} and \cite{hsing-1996} proved the asymptotic independence of $\sum_{i=1}^n u_i$ and $\max_{i=1,\ldots,n} u_i$ without requiring independent and identical distributions. More specifically, $\sum_{i=1}^n u_i$ and $\max_{i=1,\ldots,n} u_i$ are asymptotically independent when $u_1,\ldots,u_n$ are strongly mixing stationary random variables or stationary normal random variables.

\smallskip
\textbf{Remark 3}:  However, results in \cite{hsing-1995} and \cite{hsing-1996} did not directly apply to the asymptotic independence of $S_n$ and $L_n$ in high-dimensional independence tests. $S_n$ and $L_n$ are U statistics, and their dependencies among summations or maximum do not fall into the general local dependence category. They are nonlinear random variables \citep{chenshao-2007}. To the best of our knowledge, our work is the first in the statistics and econometrics literatures to obtain this surprising asymptotic independence.


\smallskip

Combining Lemmas \ref{lawSn}--\ref{lawLn} and Theorem \ref{birthday}, we obtain the intermediate joint limiting distribution of $S_n$ and $L_n$ with the explicit rate of convergence in the following theorem.

\begin{theorem}\lbl{prasing}  Under the same conditions of Theorem \ref{birthday}, we have
\be
   \sup_{y, z}\Big|P_{{S_n,L_n}}(z,y) -\Phi(z) (1-F(y))\Big|\leq C\max\Big( \min(p^{-1/5}, \sqrt{n/p}), \sqrt{\frac{(\log p)^3}{n}}\Big).
\ee
 \end{theorem}

A direct application of Theorem \ref{prasing} results in the final joint limiting law of $S_n$ and $L_n$.

\begin{coro}\label{prasing2}
Under the same conditions of Theorem \ref{birthday}, as $n,p\rightarrow \infty$, we have
$$
\sup_{y, z}\Big|P_{{S_n,L_n}}(z,y)- \Phi(z) \left(1-e^{\frac{-1}{\sqrt{8\pi}}e^{\frac{-y}2}}\right)\Big|\longrightarrow 0.
$$
\end{coro}

As shown in Corollary \ref{prasing2}, the final joint limiting distribution of $S_n$ and $L_n$ leads to slow convergence. Following \cite{hall-1979} and \cite{liulinshao-2008}, we recommend the intermediate limiting distribution (i.e., convolution of distributions $\Phi(z)$ and $F(y)$) in the high-dimensional independence testing, instead of their final limiting laws.

\section{Model-Free Joint Limiting Law}
\lbl{jointlaw2}

In this section, we shall study the model-free joint limiting law of the quadratic type statistic and the extreme-value type statistic. Similarly as in Section \ref{jointlaw1}, we first establish their surprising asymptotic independence, and then derive their intermediate joint limiting distribution with the explicit rate of uniform convergence under the high-dimensional setting.

Assume that  $X_1, \ldots, X_p$ are the columns form $\mathcal{X}_{n\times p}$, and all the entries of $\mathcal{X}_{n\times p}$ are continuous and independent random variables. Let $F_{ki}$ denote the rank of $x_{ki}$ in $X_i$, and $N_{ki}=\sqrt{\frac{12}{n^2-1}} (F_{ki}-\frac{n+1}{2})$ be the normalized rank. Define $N_{i}=(N_{1i}, \ldots, N_{ni})'$. For any $1\leq i, j\leq p$, the Spearman rank correlation coefficient between $X_i$ and $X_j$ is
\bea\lbl{rcorr}
r_{ij}=\frac{1}{n}\sum_{k=1}^nN_{ki}N_{kj}
\eea
Define $\mathbf{R}=(r_{ij})_{p\times p}$ as the sample Spearman rank correlation matrix.

Now, we introduce the quadratic type statistic $T_n$ and the extreme-value type statistic $M_n$ as follows:
\be
T_n=\sum_{1\leq i< j\leq p} r_{ij}^2,
\ee
which is the sum of squares of the off-diagonal entries of $\mathbf{R}$, and
\be\lbl{mount1}
M_n=\max_{1\leq i <  j \leq p}|r_{ij}|,
\ee
which is the largest magnitude of the off-diagonal entries of $\mathbf{R}$. Next, we define the marginal distribution functions of $T_n$ and $M_n$ as
$$
P_{T_n}(z)=P\big(\beta_n(nT_n-\alpha_n)\leq z\big)
$$
where $\alpha_n$ and $\beta_n$ are specified in Lemma \ref{lawTn}, and
$$P_{M_n}(y)=P\big(n M_{n}^2 - 4\log p + \log\log p \le y \big).$$
Moreover, we introduce their joint distribution function as
$$
P_{{T_n,M_n}}(z,y)=P\big(\{\beta_n(nT_n-\alpha_n)\leq z\}\cap \{n M_{n}^2 - 4\log p + \log\log p \ge y\} \big).
$$
Remark that $P_{{T_n,M_n}}(z,y)$ computes the probability of $\{M_{n}^2 - 4\log p + \log\log p \ge y\}$ instead of $\{M_{n}^2 - 4\log p + \log\log p \le y\}$. Again, it facilitates derivation of the explicit joint limiting distribution. In the sequel, after studying the limiting behaviours of marginal distribution functions $P_{T_n}(z)$ and $P_{M_n}(y)$, we derive their joint limiting law $P_{{T_n,M_n}}(z,y)$. Especially, we do not impose any model assumption such as Condition (C) to obtain the limiting laws.

\cite{bao-etal-2015} has studied the asymptotic distribution of $T_n$, and \cite{zhou-2007} and \cite{han2014} worked on the asymptotic distribution of $M_n$. $T_n$ asymptotically converges to a normal distribution, while $M_n$ converges to the same Gumbel distribution as $L_n$. None of them derive the explicit convergence rate under the (ultra) high-dimensional setting. Now we investigate the model-free asymptotic distributions of $T_n$ and $M_n$ in Lemmas \ref{lawTn} and \ref{lawMn}. In particular, we obtain their explicit uniform rates of uniform convergence under the ultra high-dimensional setting.

\begin{lemma}\lbl{lawTn} Suppose that $n,p\to \infty$. Then, we have
\begin{eqnarray}\lbl{stupid1}
 \sup_t |P( \beta_n(nT_n-\alpha_n) \leq t )-\Phi(t)| \leq C(\min(n^2, p))^{-1/5}.
 \end{eqnarray}
Here, $\alpha_n={p \choose 2}(1+\frac{1}{n-1})$ and
$  \beta_n=n/({p\choose 2}\sqrt{V_n'}) ,$
 where    \begin{eqnarray*}
 V_n'=\frac{4n^2EN_{11}^2N_{21}^2+2n(EN_{11}^4-1)^2-4n}{p(p-1)} +O(\frac{1}{p^2}).
  \end{eqnarray*}
\end{lemma}

\begin{lemma}\label{lawMn}
Suppose that $n,p\to \infty$ and $\log p=o(n^{1/3})$. Then, we have
\begin{eqnarray}
\sup_{y}\Big| P\left(n M_{n}^2 - 4\log p + \log\log p \le y\right)-F(y)\Big|\leq C\sqrt{\frac{(\log p)^3}{n}}
\end{eqnarray}
where $F(y)=\exp\Big(-\frac{p^2-p}{2}P(\chi^2(1)\geq 4\log p - \log\log p + y)\Big)$.
\end{lemma}

\textbf{Remark 4}: $M_n$ is the rank-based extreme-value type statistic, and it would also suffer from size distortions due to their slow convergence to the Gumbel distribution $e^{{-1}/{(\sqrt{8\pi})}e^{-{y}/2}}$. We again follow the idea of \cite{liulinshao-2008} to use its intermediate limiting distribution $F(y)$, and obtain the explicit rate of uniform convergence in high dimensions. Section 5 demonstrates the numerical performance of the intermediate limiting distribution.

\smallskip
Given Lemmas \ref{lawTn} and \ref{lawMn}, we are ready to prove the surprising asymptotic independence of $T_n$ and $M_n$ in Theorem \ref{birthdaycake}, and also obtain the explicit rate of uniform convergence.

 \begin{theorem}\lbl{birthdaycake}  Assume that $n,p\to \infty$ and $\log p=o(n^{1/6})$ as $n\to\infty.$ Then
\be
\sup_{y, z}\Big|P_{{T_n,M_n}}(z,y)-P_{T_n}(z) (1-P_{M_n}(y))\Big|\leq C  \min(n^2, p)^{-1/5}.
\ee
 \end{theorem}

Here, $T_n$ is the {sum} of dependent variables $\{r_{ij};\, 1 \leq i<j\leq n\}$, and $M_n$ is the {maximum} of these dependent variables. Recall that $r_{ij}=\frac{1}{n}\sum_{k=1}^nN_{ki}N_{kj}$. Since $\{N_{ki};\, 1 \leq k\leq n\}$ are not independent, Theorem \ref{birthdaycake} encountered more technical difficult than Theorem \ref{birthday}. We ought to use different techniques such as martingales to prove the asymptotic independence.


\smallskip
\textbf{Remark 5}: Theorem \ref{birthdaycake} presents the model-free joint limiting law, and it does not require the moment condition (i.e., Condition (C)) in Theorem \ref{birthday}. However, with the aid of Condition (C), Theorem \ref{birthday} achieves an improved rate of uniform convergence.

\smallskip

Combining Lemmas \ref{lawTn}--\ref{lawMn} and Theorem \ref{birthdaycake}, we obtain the intermediate joint limiting distribution of $T_n$ and $M_n$ with the explicit rate of convergence in the following theorem.

 \begin{theorem}\lbl{rasing}  Under the same conditions of Theorem \ref{birthdaycake}, we have
   \begin{eqnarray*} &&\sup_{y, z}\Big|P_{{T_n,M_n}}(z,y) -\Phi(z) (1-F(y))\Big|\leq C\max\Big( \min(n^2, p)^{-1/5}, \sqrt{\frac{(\log p)^3}{n}}\Big),
 \end{eqnarray*}
 where $F(y)=\exp\big(-\exp(-(p^2-p)(1-\Phi(4\log p - \log\log p + y))\big)$.
 \end{theorem}

A direct application of Theorem \ref{rasing} results in the final joint limiting law of $T_n$ and $M_n$.

\begin{coro}\label{rasing2}
Under the same conditions of Theorem \ref{birthdaycake}, as $n,p\rightarrow 0$, we have
$$
\sup_{y, z}\Big|P_{{T_n,M_n}}(z,y)- \Phi(z) \left(1-e^{\frac{-1}{\sqrt{8\pi}}e^{\frac{-y}2}}\right)\Big|\longrightarrow 0.
$$
\end{coro}

As shown in Theorem \ref{rasing} and Corollary \ref{rasing2}, we recommend the intermediate limiting distribution (i.e., convolution of distributions $\Phi(z)$ and $F(y)$) in the model-free high-dimensional independence testing, as the final limiting distribution may suffers from size distortion due to its slow convergence \citep{hall-1979,liulinshao-2008}.

\section{High-Dimensional Independence Tests}

The joint limiting laws derived in Sections 2--3 have important statistical applications to test the covariance structure of a high dimensional random variable. In Section 4 we propose new high-dimensional independence tests based on joint limiting laws, addressing the null hypothesis $\mathbf{H}_0: x_1, \ldots, x_p$ {are mutually independent}. Following \cite{caijiang-2011}, our proposed statistics can be extended to test bandedness in high dimensions.

To begin with, we introduce two parameter spaces for the covariance matrix $\Sigma$:
 \begin{eqnarray*}
 \mathcal{G}_1&=&\left\{\bSigma=(\sigma_{ij})_{p\times p}: \ \sigma_{ij}=\sigma_{ji} \ \text{and} \ \max_{i<j}|\sigma_{ij}|>C\sqrt{\frac{\log p}{n}}\right\};\\
  \mathcal{G}_2&=&\left\{\bSigma=(\sigma_{ij})_{p\times p}: \ \sigma_{ij}=\sigma_{ji} \ \text{and} \ \sum_{i<j}\sigma^2_{ij}\gg C { p\log^2p} \right\}.
 \end{eqnarray*}
Note that $\mathcal{G}_1$ represents the sparse alternative where the covariance has a few relatively large off-diagonal entries, and $\mathcal{G}_2$ denotes the dense alternative where the covariance contains a lot of small nonzero off-diagonal entries. Existing methods employs either the sum-of-squares type statistics to test against the dense alternative or the extreme-value type statistics to test against the sparse alternative. Our work provides an innovative testing procedure to boost the power against more general alternatives. More specifically, we consider the following composite alternative hypothesis:
$$
\mathbf{H}_1: \bSigma\in \mathcal{G}_1\cup \mathcal{G}_2
$$

Motivated by asymptotic independencies in Theorems \ref{birthday} and \ref{birthdaycake} and joint limiting laws in Theorems \ref{prasing} and \ref{rasing}, we propose the following two new testing procedures:
$$
TS_n^1 =  I_{\{b_n(S_n-a_n)+(n L_{n}^2 - 4\log p + \log\log p)\geq c_\alpha\}}
$$
and
$$
TS_n^2 =  I_{\{\beta_n(T_n-\alpha_n)+(n M_{n}^2 - 4\log p + \log\log p)\geq c_\alpha\}}
$$
where the threshold $c_\alpha$ is defined as the $\alpha$ upper quantile of the convolution distribution $\Phi\star F$. Here, $TS_n^1$ is based on the sum of normalized test statistics $S_n$ and $L_n$, and $TS_n^2$ is based on the sum of normalized model-free test statistics $T_n$ and $M_n$. $TS_n^1=1$ or $TS_n^2=1$ leads to the rejection for high-dimensional independence test.

\smallskip
\textbf{Remark 6}: We may use the following empirical approximation method to effectively obtain the threshold $c_\alpha$, which also helps find the threshold for extreme-value form statistics. Theorems \ref{prasing} and \ref{rasing} provide the convergence rates for joint limiting laws to the intermediate limiting distribution, and Lemmas \ref{lawLn} and \ref{lawMn} give the convergence rates for $L_n$ and $M_n$. Notice that $O(\sqrt{(\log p)^3/n})$ could still lead to a slow convergence when $p$ is much large than $n$. In practice, we recommend using the empirical approximation technique to avoid potential size distortions and boost the numerical performance. Instead of simulating $F(y)$, we need to approximate $\lambda_n=\sum_{1\leq i<j \leq p}P(  |\hat{\sigma}_{ij}|>\sqrt{4\log p-\log\log p+y} )$ as in Lemma 5. Since $ \hat{\sigma}_{ij}$ only depend on independent pair $X_i$ and $X_j$, we can approximate the distribution  of  $P(|\hat{\sigma}_{ij}|>\sqrt{4\log p-\log\log p+y})$ using empirical processes. Now we take $L_n$ and $TS_n^1$ for example. We generate $M$ independent pairs $(X_1^{i},X_2^{i})$ for $i=1,\ldots,M$, where $X_1^{i}$ and $X_2^{i}$ are independent. Next, we can get the empirical estimation of $P(|\hat{\sigma}_{ij}|>\sqrt{4\log p-\log\log p+y})$ based on $(X_1^{1},X_2^{1}),\cdots,(X_1^{M},X_2^{M})$. By the Dvoretzky--Kiefer--Wolfowitz inequality, we have better approximation of $\lambda_n$ as long as $M\gg p^2$ holds. In this way, we may obtain the $\alpha$ upper quantile for $L_n$ or $TS^1_n$.

\smallskip
Let $P_{\mathbf{H}_0}(\cdot)$ be the probability given the null hypothesis $\mathbf{H}_0$, and $P_{\mathbf{H}_1}(\cdot)$ be the probability given the alternative hypothesis $\mathbf{H}_1$. $P_{\mathbf{H}_0}(TS_n^1=1)$ denotes the probability of erroneously rejecting $\mathbf{H}_0$, and $P_{\mathbf{H}_1}(TS_n^1=1)$ is the probability of correctly rejecting $\mathbf{H}_0$. In the sequel, we shall prove that $TS_n^1$ and $TS_n^2$ not only well control the significance level but also achieve the consistent power, namely, $P_{\mathbf{H}_0}(TS_n^1=1)\rightarrow \alpha$ and $P_{\mathbf{H}_1}(TS_n^1=1)\rightarrow 1$, or , $P_{\mathbf{H}_0}(TS_n^2=1)\rightarrow \alpha$ and $P_{\mathbf{H}_1}(TS_n^2=1)\rightarrow 1$

As shown in the following theorem, both test statistics $TS_n^1$ and $TS_n^2$ achieve the asymptotic significance level $\alpha$.

\begin{theorem}\lbl{tb0} Under the same conditions of Theorem \ref{prasing}, we have
$$
P_{\mathbf{H}_0}(TS_n^1=1)\rightarrow \alpha
\quad \text{as} \ n\rightarrow \infty.
$$
Without requiring Condition (C), under the same conditions of Theorem \ref{rasing}, we then have
$$
P_{\mathbf{H}_0}(TS_n^2=1)\rightarrow \alpha
\quad \text{as} \ n\rightarrow \infty.
$$
\end{theorem}

To derive the power, we assume that $\mathcal{X}=(\mathbf{x}_1,\cdots,\mathbf{x}_n)'$ are $n$ independent observations of the random vector ${\mathbf{x}=(x_1, \ldots, x_p)'}$ with zero means, unit variances and $\bSigma\in \mathcal{G}_1\cup \mathcal{G}_2$. Notice that $\bSigma$ is no longer a diagonal matrix, and instead, $\bSigma$ either has a few large entries or has many small entries. In the following theorem, we prove the consistent power of our proposed testing procedure $TS_n^1$ under the composite alternative hypothesis. We may need more technical analysis to prove the consistent power of $TS_n^2$, which is beyond the scope of this paper for space consideration.

\begin{theorem}\lbl{tb1} Suppose that  $\bSigma\in \mathcal{G}_1\cup \mathcal{G}_2$, and we also assume that
 \begin{align}
    & \operatorname{E} [\,x_{1i_1}x_{1i_2}\dots x_{1i_{2k}}\,] = \sum\prod \operatorname{E}[\,x_{1i}x_{1j}\,], \\
    & \operatorname{E}[\,x_{1i_1}x_{1i_2}\dots x_{1i_{2k-1}}\,] = 0,
  \end{align}
  where $ \sum\prod$ is the sum over all distinct ways of partitioning $x_{1i_1},\ldots , x_{1i_{2k}}$ into pairs and $k\leq 4$.  As ${p=O(n^3)}$, $n\to \infty$, $p\to \infty$, $\textrm{tr}(\Sigma^8)/\text{tr}^2(\Sigma^4)\to 0$ and $\max_{i}\sum_{j=1}^p\sigma_{ij}^2/\text{tr}(\Sigma^2)\to 0.$ Then, $TS_n^1$ has the consistent power that
 $$
{\inf_{\Sigma\in \mathcal{G}_1\cup \mathcal{G}_2 } P (TS_n^1=1)\to 1 }
\quad \text{as} \ n\rightarrow \infty.
 $$
 \end{theorem}

\section{Simulation Studies}

In Section 6, we shall demonstrate the numerical performance of our proposed test statistics $TS_n^1$ and $TS_n^2$. We numerically investigate them together with testing procedures based on either extreme-value type statistics $L_n$ and $M_n$ or quadratic type statistics $S_n$ and $T_n$. Now, we consider six different models to simulate the data $\mathcal{X}=(\mathbf{x}_1,\cdots,\mathbf{x}_n)'=(x_{ij})_{n\times p}$ for size and power comparisons:
\begin{description}
\item[Model 1a:]  $x_{ij}$'s are independently generated from standard normal distribution $N(0,1)$.
\item[Model 1b:] $x_{ij}$'s are independently generated from standard Cauchy distribution with the probability density function $f(x)=1/(\pi(1+x^2))$.
\item[Model 2a:] $\mathbf{x}_1,\cdots,\mathbf{x}_n$ are independently generated from a Gaussian distribution $N_{p}(0,
\Sigma)$ with $\Sigma=(\sigma_{ij})_{p\times p}$, $\sigma_{11}=\cdots=\sigma_{pp}=1$, $\sigma_{12}=2.5\sqrt{\frac{\log p}{n}}$ and $\sigma_{ij}=0$ for others.
\item[Model 2b:] $\mathbf{x}_1,\cdots,\mathbf{x}_n$ are independently generated as follows: we first generate $z_1, \ldots, z_p$ from standard Cauchy distribution, and then construct $\mathbf{x}=(x_1,\cdots,x_p)'$ by setting $x_1= z_1+ \sqrt{\frac{\log p}{n}}z_2$,  $x_2= z_2+ \sqrt{\frac{\log p}{n}}z_1$ and $x_j=z_j$ for $3\leq j\leq p$.
\item[Model 3a:] $\mathbf{x}_1,\cdots,\mathbf{x}_n$ are independently generated from a Gaussian distribution $N_{p}(0,
\Sigma)$ with $\Sigma=I_p+D_p$ and $D_p=(2\log p/p)_{p\times p}$.
\item[Model 3b:]  $\mathbf{x}_1,\cdots,\mathbf{x}_n$ are independently generated as follows: we first generate $z_1, \ldots, z_p$ from standard Cauchy distribution, and then construct $\mathbf{x}=(x_1,\cdots,x_p)'$ by setting $x_j=z_j+\frac{1}{10p}\sum_{i\neq j}z_i$ for $1\leq j\leq p$.
\end{description}

For each simulated model, we let $n=200$ and $p=50, 100, 200, 400, 600, 800,$ $\& 1000$. Models 1a and 1b mimic the null hypothesis $\mathbf{H}_{0}$ to examine the size. To compare the test power, Models 2a and 2b consider the sparse alternative hypothesis $\mathbf{H}_{1}: \bSigma\in \mathcal{G}_1$, and Models 3a and 3b consider the dense alternative hypothesis $\mathbf{H}_{1}: \bSigma\in \mathcal{G}_2$.  We compare the numerical results for $L_n$, $M_n$, $S_n$, $T_n$, $TS_n^1$ and $TS_n^2$ in Models 1a, 2a and 3a, since Condition (C) is obviously satisfied by the normal data. On the other hand, as the Cauchy data violates Condition (C), we only compare $M_n$, $T_n$, and $TS_n^2$ in Models 1b, 2b and 3b. We employ the limiting distribution (i.e., $N(0,1)$) to obtain the threshold for $S_n$ and $T_n$, and use the empirical approximation of intermediate limiting distributions to obtain the threshold $c_{\alpha}$ for $L_n$, $M_n$, $TS_n^1$ and $TS_n^2$ as discussed in Remark 6. Moreover, we specify the significance level $\alpha=0.05$ for each test. For each simulation model, we simulate $1000$ independent datasets, and each dataset consists of $2000$ independent samples.

Simulation results are summarized in Tables 1--3, showing the average percentages of rejections out of $1000$ independent repetitions. Based on Tables 1--3, we have the following observations. For the normal data, compared with $L_n$ and $S_n$, $TS_n^1$ still achieves a favorable size in Model 1a. $L_n$ clearly suffers from low power against dense alternatives, and $S_n$ suffers from low power against sparse alternatives. However, $TS_n^1$ retains the nice power against both the sparse alternative in Model 2a and the dense alternative in Model 3a, which is very appealing to real-world analysis. Moreover, those rank-based testing procedures (i.e., $M_n$, $T_n$, and $TS_n^2$) are slightly outperformed by $S_n$, $L_n$, and $TS_n^1$ respectively. For the Cauchy data, $TS_n^2$ also achieves the favorable size and the best power against both sparse and dense alternatives. In a summary, our proposed test statistics $TS_n^1$ and $TS_n^2$ work best for high-dimensional independence tests with more general alternatives.

 \begin{table}[ht]
\caption{Performance of testing independence in simulated models 1a and 1b.} 
\vspace{.1in}
\centering 
\begin{tabular}{ c c c c c c c c c c c c  } 
\hline\hline 
 &\multicolumn{6}{c}{Model 1a} & \multicolumn{3}{c}{Model 1b} \\
\cline{2-7}  \cline{9-11}
$p$ &$S_n$ &$L_n$ & $TS^1_n$& $T_n$&$M_n$ &$TS^2_n$&  &$T_n$&$M_n$ &$TS^2_n$   \\ [0.5ex] 
\hline 
 50 & 0.0533 & 0.0518& 0.0607 & 0.0532&0.0523&  0.0613& &0.0531&  0.0523&0.0619 \\ 
  100 & 0.0524& 0.0554&  0.0604& 0.0525&0.0549& 0.0609&&0.052&0.0546&0.0607\\
     200 & 0.0516 & 0.0541&  0.0621 & 0.0515&0.0635&  0.066&&0.0518&0.0619&0.0657\\
 400 & 0.0513& 0.0510&  0.0571& 0.0515&0.0598& 0.0636&&0.0514& 0.0615&0.0639\\
  600 & 0.0513 & 0.0535&  0.0537& 0.0513&0.0675& 0.0676&&0.0514&0.0672&0.0673\\
  800 & 0.0517 & 0.0592&0.0518& 0.0516&0.0739&0.0666&&0.0512&0.0736& 0.0659\\
  1000&  0.0512& 0.0586& 0.0487 &0.0516&0.0845&0.0653&&0.0517&0.0856&0.0646\\
 \hline
 \hline
\end{tabular}
\end{table}

   \begin{table}[ht]
\caption{Performance of testing independence in simulated models 2a and 2b.} 
\vspace{.1in}
\centering 
\begin{tabular}{ c c c c c c c c c c c c  } 
\hline\hline 
 &\multicolumn{6}{c}{Model 2a} & \multicolumn{3}{c}{Model 2b} \\
\cline{2-7}  \cline{9-11}
$p$ &$S_n$ &$L_n$ & $TS^1_n$& $T_n$&$M_n$ &$TS^2_n$&  &$T_n$&$M_n$ &$TS^2_n$   \\ [0.5ex] 
\hline
 50 & 0.1287 & 0.8967& 0.8955 & 0.1205&0.8289&  0.8297& &0.1701& 0.9842&0.9835 \\ 
 100 & 0.0899& 0.9349&  0.9306& 0.0868&0.8745& 0.8708&&0.1044&0.9863&0.9851\\
     200 & 0.0717 & 0.9590&  0.9579& 0.0704&0.9162& 0.9103&&0.0761& 0.9872&0.9861\\
      400 & 0.0619 & 0.9778&  0.9752& 0.0617&0.9429& 0.9387&&0.0638&0.9883&0.9864\\
    600 & 0.0588& 0.9851&  0.9826& 0.0581&0.9589& 0.9543&&0.0593&0.9883&0.9874\\
    800 & 0.0571 & 0.9890&0.9863& 0.0568&0.9682&0.9623&&0.0577&0.9902& 0.9885\\
 1000&  0.0561& 0.9912& 0.9892 &0.0568&0.9755&0.9679&&0.0570&0.9912&0.9883\\
\hline
\end{tabular}
\end{table}

 \begin{table}[ht]
\caption{Performance of testing independence in simulated models 3a and 3b.} 
\vspace{.1in}
\centering 
\begin{tabular}{ c c c c c c c c c c c c  } 
\hline\hline 
 &\multicolumn{6}{c}{Model 2a} & \multicolumn{3}{c}{Model 2b} \\
\cline{2-7}  \cline{9-11}
$p$ &$S_n$ &$L_n$ & $TS^1_n$& $T_n$&$M_n$ &$TS^2_n$&  &$T_n$&$M_n$ &$TS^2_n$   \\ [0.5ex] 
\hline
 50 & 1 &  0.8896& 1 & 1&0.8582 & 1& &0.9936 & 0.4959&0.968 \\ 
  100 & 1 & 0.5331& 1& 1&0.4732& 0.9994&&0.9987&0.5078&0.9886\\
  200 & 1& 0.2245&  0.9993& 1&0.2217& 0.9993&&0.9994&0.5548&0.9968\\
   400 & 0.9993 & 0.1055&  0.9982 & 0.9993&0.1145& 0.9972&&0.9991&0.5759&0.9993\\
  600 & 0.9994 & 0.0846&  0.9926& 0.9997&0.1103& 0.9873&&0.9990&0.6122&0.9997\\
  800 & 0.9990 & 0.0808&0.9763 & 0.9999&0.0997&0.9645&&0.9994&0.6429& 0.9995\\
  1000&  0.9991& 0.0773& 0.9441&0.9984&0.1059&0.928&&1& 0.6852&0.9999\\
 \hline
 \hline
\end{tabular}
\end{table}

\section{Discussion}
\label{disc}

Under the ultra high-dimensional setting, our work proposes two novel testing procedures with explicit limiting distributions to boost the power against general alternatives. Surprisingly, we prove that extreme-value form and quadratic form statistics are asymptotically independent. Our work provides very important insights to high-dimensional independence tests, and fills an important gap. To the best of our knowledge, this is the first work in the statistics literature to derive the asymptotic independence of the quadratic form statistics and the extreme-value form statistics in high-dimensional inference.

Motivated by this work, there are many interesting research questions to explore in high-dimensional inference. For example, it is very interesting to study the joint limiting laws of quadratic form test statistic \citep{lichen-2012} and extreme-value form test statistic \citep{caixia-2013} in the two-sample covariance matrix testing problem. Moreover, as shown in \cite{caima-2013}, it is important to find the optimal hypothesis testing procedure for high dimensional covariance matrices against more general alternatives.

\section{Proofs}
\label{proof.sec}


We first cite two lemmas as our technical tools, which are needed subsequently in the proofs.  Lemma \ref{stein} gives a Poisson approximation result, which is essentially a special case of Theorem 1 from \cite{arratia-etal-1989}. Lemma \ref{shao} gives a Berry-Esseen type large deviation bounds, i.e., Example 1 from \cite{sakhanenko1991}. See also Lemma 6.2 from Liu et al (2008).

\begin{lemma}\label{stein} Let $I$ be an index set and $\{B_{\alpha}, \alpha\in I\}$ be a set of subsets of $I,$ that is, $B_{\alpha}\subset I$ for each $\alpha \in I.$  Let also $\{\eta_{\alpha}, \alpha\in I\}$ be random variables. For a given $t\in \mathbb{R},$ set $\lambda=\sum_{\alpha\in I}P(\eta_{\alpha}>t).$ Then
\beaa
|P(\max_{\alpha \in I}\eta_{\alpha} \leq t)-e^{-\lambda}| \leq (1\wedge \lambda^{-1})(b_1+b_2+b_3)
\eeaa
where
\beaa
& & b_1=\sum_{\alpha \in I}\sum_{\beta \in B_{\alpha}}P(\eta_{\alpha} >t)P(\eta_{\beta} >t),\\
& & b_2=\sum_{\alpha \in I}\sum_{\alpha\ne \beta \in B_{\alpha}}P(\eta_{\alpha} >t, \eta_{\beta} >t),\\
 & & b_3=\sum_{\alpha \in I}E|P(\eta_{\alpha} >t|\sigma(\eta_{\beta}, \beta \notin B_{\alpha})) - P(\eta_{\alpha} >t)|,
\eeaa
and $\sigma(\eta_{\beta}, \beta \notin B_{\alpha})$ is the $\sigma$-algebra generated by $\{\eta_{\beta}, \beta \notin B_{\alpha}\}.$
In particular, if $\eta_{\alpha}$ is independent of $\{\eta_{\beta}, \beta \notin B_{\alpha}\}$ for each $\alpha,$ then $b_3=0.$
\end{lemma}

\begin{lemma}\lbl{shao} Let $\xi_i, 1\leq i \leq n,$ be independent random variables with $E\xi_i=0.$ Put
\begin{eqnarray*}
s_n^2=\sum_{i=1}^nE\xi_i^2,\ \ \ \varrho_n=\sum_{i=1}^nE|\xi_i|^3,\ \ \ S_n=\sum_{i=1}^n\xi_i.
\end{eqnarray*}
Assume $\max_{1\leq i \leq n}|\xi_i| \leq y_ns_n$ for some $0<y_n \leq 1.$ Then
\begin{eqnarray*}
P(S_n\geq xs_n) =e^{\gamma_n(x/s_n)}(1-\Phi(x))(1+\theta_{n,x}(1+x)s_n^{-3}\varrho_n)
\end{eqnarray*}
for $0<x \leq 1/(18y_n),$ where $|\gamma_n(x)| \leq 2x^3\varrho_n$ and $|\theta_{n,x}|\leq 36.$
\end{lemma}

\subsection{Proof of Lemma \ref{lawSn}}

\begin{proof}[Proof of Lemma \ref{lawSn}]

Our proof is divided into two parts. In the first part, we derive the rate of convergence of central limit theorem when $p\gg n^{5/3}$. Note that $nS_n=\sum_{1\le i<j\le p} n\hat{\sigma}_{ij}^2$ is the sum of  nonlinear statistics. Motivated by the concentration inequality approach in \cite{chenshao-2007}, we define that $g(X_i)=E(n\hat{\sigma}_{ij}^2|X_i)-1=\frac{1}{n}\sum_{k=1}^n(x_{ki}^2-1)$ and $\sigma_i^2=E\{g^2(X_i)\}= \frac{1}{n}(Ex_{1i}^4-1)$. It is not difficult to check that $\text{Var}(n\hat{\sigma}_{ij}^2)=\frac{Ex_{1i}^4Ex_{1j}^4-3}{n}+2.$ Now, given $p\gg n^{5/3}$, we apply Theorem 3.2 of Chen and Shao (2007) to obtain that
 \begin{eqnarray}\lbl{CS}
\sup_{z}\Big|P\big( b_n(nS_n-a_n)\leq z\big)-\Phi(z)\Big|\leq \frac{C\sqrt{n}}{\sqrt{p}},
 \end{eqnarray}
where  $a_n={p\choose 2}$ and $b_n=\sqrt{n}/({(p-1)\sqrt{\sum_{i=1}^p(Ex_{1i}^4-1)}})$.

In the second part, we use the martingale method of \cite{haeusler-1988} to show the rate of convergence for the general case. Define $\mathcal{F}_0=\{\emptyset, \Omega\}$, $\mathcal{F}_{k}=\sigma\{X_1,\ldots,  X_k\}$, $k=1,2,\ldots, p$. Let $E_k$ denote the conditional expectation of given $\mathcal{F}_k$. Define $Z_n=n^2S_n/{ p \choose 2}$. Then, we write $Z_n-EZ_n=\sum_{k=1}^{p}M_{k}$, where $M_{k}=(E_k-E_{k-1})Z_n$ and $E_{k}=E(\cdot | \mathcal{F}_k)$. Then for every $p$, $\{M_{k}, 1\leq k\leq p\} $ is a martingale difference sequence with respect to ${\mathcal{F}_k, 1\leq k\leq p}$.
By using the martingale central limit theorem \citep{haeusler-1988}, we know
$$
 \sup_t |P(\frac{Z_n-EZ_n}{\sqrt{\text{Var}(Z_n)}}\leq t )-\Phi(t)| \leq C E\bigg(\frac{\sum_{k=1}^pE_{k-1}(M_{k}^2)}{\text{Var}{(Z_n)}}-1\bigg)^{2/5} +C \bigg(\frac{\sum_{k=1}^pE (M_{k}^4)}{\text{Var}^2{(Z_n)}}\bigg)^{1/5}.
$$
Now it only remains to bound the following two terms:
\begin{eqnarray}\lbl{smelly}
E\bigg(\frac{\sum_{k=1}^p E_{k-1}(M_{k}^2)}{\text{Var}{(Z_n)}}-1\bigg)^2\,   \mbox{and}\, \frac{\sum_{k=1}^pE M_{k}^4}{\text{Var}^2{(Z_n)}}.
 \end{eqnarray}
First we bound the first term of (\ref{smelly}). To this end, we bound $\text{Var}{(Z_n)}$ and $\text{Var}(E_{k-1}(M_{ k}^2))$ respectively. Recall that $Z_n-EZ_n=\sum_{k=1}^{p}M_{k}$. It is not difficult to find that $$M_{k}=\frac{2}{p(p-1)}\sum_{i=1}^{k-1}\{(X_k'X_i)^2-X_k'X_k-X_i'X_i+n\}+\frac{2}{p}(X_k'X_k-n).$$ Given the explicit form of $M_{k}$, we may derive the conditional expectation $E_{k-1}(M_{ k}^2)$ as
\begin{eqnarray*}
&& E_{k-1}(M_{ k}^2)\\&=&\frac{8}{p^2(p-1)^2}  \sum_{i=1}^{k-1}\sum_{j=1}^{k-1} \{(X'_iX_{j})^2-X_{i}'X_i -X^{T}_jX_j+n\}\\&&+\frac{4(Ex_{1k}^4-3)}{p^2(p-1)^2} \sum_{i=1}^{k-1}\sum_{j=1}^{k-1} \sum_{m=1}^n (x_{mi}^2-1)(x_{mj}^2-1)\\&&+\frac{8(Ex_{1k}^4-1)}{p^2(p-1)}  \sum_{i=1}^k\sum_{m=1}^n (x_{mi}^2-1)+
 \frac{4n(Ex_{1k}^4-1)}{p^2}\\
& := &\frac{8}{p^2(p-1)^2} A_k+\frac{4(Ex_{1k}^4-3)}{p^2(p-1)^2}B_k+ +\frac{8(Ex_{1k}^4-1)}{p^2(p-1)} C_k+ \frac{4n(Ex_{1k}^4-1)}{p^2}
 \end{eqnarray*}
and also the expectation $E (M_{k}^2)=E(E_{k-1}(M_{ k}^2))$ by using the law of iterated expectation:
$$
E (M_{k}^2)=\frac{4(Ex_{1k}^4-1)n\sum_{i=1}^{k-1}(Ex_{1i}^4-1)+8n(n-1)(k-1)  }{p^2(p-1)^2}   + \frac{4n(Ex_{ik}^4-1)}{p^2}.
$$
Then, it immediately implies that
 $$\text{Var}(Z_n)=\frac{4n(n- 1)}{p(p-1)}+\frac{4n\sum_{1\leq i<j\leq p}(Ex_{1i}^4-1)(Ex_{1j}^4-1)}{p^2(p-1)^2}+\frac{4n\sum_{i=1}^p(Ex_{1i}^4-1)}{p^2}.$$
Thus, we have $$\text{Var}(Z_n)=C\frac{n^2}{p^2}+C\frac{n}{p}.$$

Note that
$$\text{Var}(\sum_{k=1}^pE_{k-1}(M_{ k}^2))=\sum_{k=1}^p (2p-2k+1)\text{Var}(E_{k-1}(M_{ k}^2)),$$
Then, $\text{Var}(E_{k-1}(M_{ k}^2))\leq C[\frac{1}{p^4(p-1)^4}\{\text{Var}(A_k)+\text{Var}(B_k)\}+\frac{1}{p^4(p-1)^2}\text{Var}(C_k)].$
Moreover, we have $\text{Var}(A_k)=Ckn^3+k^2n^2$, $\text{Var}(B_k)=Ck^2n$ and $\text{Var}(C_k)=Ckn $.
Hence, we obtain
$$
E\bigg(\frac{\sum_{k=1}^pE_{k-1}(M_{ k}^2)}{\text{Var}{(Z_n)}}-1\bigg)^2\leq C( \frac{1}{n^2}+\frac{1}{p}).
$$

Next we bound the second term of (\ref{smelly}). We may follow the proof of \cite{chen-etal-2010} to show that
$
 \sum_{k=1}^pE (M_{k}^4)=C\frac{n^2}{p^3}+C\frac{n^4}{p^5}.
$
Thus, we
$$\frac{\sum_{k=1}^pE (M _{k}^4)}{\text{Var}^2{(Z_n)}}\leq C\frac{1}{p}.$$

By noting that $V_n=\text{Var}{(Z_n)}$, the proof of Lemma 1 is complete.
\end{proof}

\subsection{Proof of Lemma 2}

\begin{proof}[{Proof of Lemma 2}]
Let $\alpha_p=\sqrt{4\log p-\log\log p+y} $.
We consider three regions respectively to obtain the uniform convergence rate: $y\geq 2\log p$, $-2\log\log p^{\theta}<y<2\log p$, and $y\leq -2\log\log p^{\theta}$. The convergence rate for $y\geq 2\log p$ or $y\leq -2\log\log p^{\theta}$ is similarly proved as in Theorem 1. Thus, we only need to consider the case when $-2\log\log p^{\theta}\leq y\leq 2 \log p$, where $\theta=4\sqrt{\pi}.$ Now, with $I:=\{(j,k): 1\leq j<k\leq d \}$ and $B_a=\{(s,t)\in I: (s,t)\neq (j, k), (s,t)\cap (i,j)\neq \emptyset\}$ for $a=(i, j)\in I$, we can use Lemma \ref{stein} to obtain the following uniform convergence rate,
\begin{eqnarray*}\sup_{-2\log\log p^{\theta}\leq y\leq 2 \log p}\Big|P_{L_n}(y)-\exp\big(-\frac{p^2-p}{2}P(
|X_i^TX_j|/\sqrt{n}\geq \alpha_p )\big)\Big|=O(\frac{1}{p^{-1+\epsilon}})\\   \end{eqnarray*}
where we have used the simple facts that $b_3=0$ and $b_1+b_2\leq O(\frac{1}{p^{-1+\epsilon}}).$

We define $h_n=C_0 \log p$ with $C_0=6/t_0$. Let $y_m=x_{mi}x_{mj}I(|x_{mi}x_{mj}|\leq h_n)$  and $\mu=Ey_m.$ It is not difficult to show that $P(|x_{mi}x_{mj}|\geq h_n)=O(p^{-6})$.   Then we have $$\Big|P( |X_i^TX_j|/\sqrt{n}\geq \alpha_p )- P(|\sum_{m}(y_m-\mu)|/\sqrt{n}\geq   \alpha_p \pm \sqrt{(\log p)^3/n})\Big|=O(p^{-4}).$$
Now by using Lemma \ref{shao}, we have
$$P(|\sum_{m}(y_m-\mu)|/\sqrt{n}\geq   \alpha_p \pm \sqrt{(\log p)^3/n})=(1-\Phi(\alpha_p \pm \sqrt{(\log p)^3/n}))(1+C\sqrt{\frac{(\log p)^3}{n}}).$$
Then
\begin{eqnarray*}\sup_{-2\log\log p^{\theta}\leq y\leq 2 \log p}\Big|\exp\big(-\frac{p^2-p}{2}P( |X_i^TX_j|/\sqrt{n}\geq \sqrt{4\log p-\log\log p+y} )\big)\\-\exp\big(-(p^2-p)(1-\Phi(\sqrt{4\log p-\log\log p+y}))\big)\Big|\leq C\frac{(\log p)^{3/2}}{\sqrt{n}}.   \end{eqnarray*}Therefore, the proof of Lemma 2 is complete.
\end{proof}

\subsection{Proof of Theorem 1}

To simplify notation, we define $y_n=4\log p - \log\log p + y$, $t_p=4\log p - \log\log p -2\log\log p^{\theta}$, and $\theta=4\sqrt{\pi}$. Before proceeding to proof of Theorem 1, we first present an important lemma, whose proof is given in \cite{lixue-2015}.

\begin{lemma}\lbl{xia3} Rearrange the two dimensional indices $\{(i,j): 1\leq i<j\leq p\}$ in any ordering and set them as $\{(i_l, j_l): 1\leq l\leq q={p \choose 2}\}.$ Define $V_l=(X_{i_l}'X_{j_l})^2/n $. Under the same conditions of Theorem \ref{birthday}, let ${d=O(\log p)}$, we have  \beaa
 \sum_{1\leq l_1<\ldots< l_d\leq q}  P(\cap_{t=1}^d \{V_{l_t}\geq 4\log p - \log\log p + y\})\leq(\frac{e^{-y/2+3}}{d})^d.\eeaa
 \end{lemma}

Now we are ready to prove Theorem 1 as follows.

\begin{proof}[Proof of Theorem 1]
The proof is divided into two parts. To obtain the uniform convergence rate for $(y,z)\in\mathbb{R}^2$, we derive the rates of convergence for three regions respectively: $(y\geq 2\log p, z)$, $(-2\log\log p^{\theta}<y<2\log p, z)$, and $(y\leq -2\log\log p^{\theta}, z)$.

In the first part, we consider $(y\geq 2\log p, z)$ and $(y\leq -2\log\log p^{\theta}, z)$ together. To begin with, we introduce a simple but useful fact: for any two events $A$ and $B$, we have
$$|P(A\cap B)-P(A)P(B)|\leq \min(P(A), P(B), P(A^c), P(B^c)).$$
Using this fact, we then have
\be \lbl{qusia}
\sup_{y\leq -2\log\log p^{\theta}, z}|P_{{S_n,L_n}}(z,y)-P_{{S_n}}(z) (1-P_{{L_n}}(y))|\leq\sup_{y\leq -2\log\log p^{\theta}}P_{{L_n}}(y)
\ee
and
\be \lbl{qusib}
\sup_{y\geq 2\log p, z}|P_{{S_n,L_n}}(z,y)-P_{{S_n}}(z) ( 1-P_{{L_n}}(y)) |\leq\sup_{y\geq 2\log p}(1-P_{{L_n}}(y)).
\ee
Using the union bound, we can further bound $\sup_{y\geq 2\log p}P_{{L_n}}(y)$ as follows:
$$\sup_{y\geq 2\log p}{(1-P_{{L_n}}(y))} \leq  p^2 P\Big(\frac{(X_{1}'X_{2})^2}{n}\geq 6\log p-\log\log p \Big)\leq Cp^{-1}.$$
To bound $\sup_{y\leq -2\log\log p^{\theta}}P_{{L_n}}(y)$, we note that $\sup_{y\leq -2\log\log p^{\theta}}P_{{L_n}}(y) \leq P_{{L_n}}( -2\log\log p^{\theta})$ due to the monotonicity of $P_{{L_n}}(\cdot)$. By using Lemma \ref{stein}, we have
$$
P_{{L_n}}( -2\log\log p^{\theta})\leq e^{-\psi_n}+Cp^3P\big(\frac{(X_{1}'X_{2})^2}{n}\geq t_p \big)^2+Cp^3P\big(\frac{(X_{1}'X_{2})^2}{n}\geq t_p,\frac{(X_{1}'X_{3})^2}{n}\geq t_p \big),
$$
where $\psi_n=\frac{p^2-p}{2}P\Big(\frac{(X_{1}'X_{2})^2}{n}\geq t_p \Big).$ Next, by using Lemmas 6.7 and 6.9 of Cai and Jiang (2011), we know that
 $$
 P\Big(\frac{(X_{1}'X_{2})^2}{n}\geq t_p,\frac{(X_{1}'X_{3})^2}{n}\geq t_p \Big)=O(p^{-4+\epsilon})$$
 and
 $$
 P\Big(\frac{(X_{1}'X_{2})^2}{n}\geq t_p \Big)\sim \frac{e^{-t_p/2}(\log p)^{-1/2}}{\sqrt{8\pi}}.
 $$
 Hence, it immediately implies that
 $$\sup_{y\leq -2\log\log p^{\theta}}P_{{L_n}}(y) \leq Cp^{-1+\epsilon}.$$
 Combining (\ref{qusia}), (\ref{qusib}) and this upper bound, we have
 \be\lbl{qusil} \sup_{y\geq 2\log p \ \text{or} \ y\leq -2\log\log p^{\theta}, \ z}|P_{{S_n,L_n}}(z,y)-P_{{S_n}}(z) {(1-P_{{L_n}}(y))}|\leq Cp^{-1+\epsilon}.
\ee

 In the second part, we consider the remaining case $(-2\log \log p^{\theta}< y< 2\log p, z)$.
For ease of notation, we rearrange the ${p \choose 2}$ distinct indices $\{(i,j): 1\leq i<j\leq p\}$ in any ordering such that $W=\{(i_l, j_l): 1\leq l\leq q={p \choose 2}\}.$ Define $V_l=(X_{i_l}'X_{j_l})^2/n $. Then, we have
$$
S_n=\sum_{l=1}^q V_l \quad \text{and} \quad L_n=\max_{l=1,\cdots,q}V_l.
$$
Equivalently, the joint distribution $P_{{S_n,L_n}}(z,y)$ can be written as
\beaa P_{{S_n,L_n}}(z,y)
&=& P\big(\{\max_{l=1,\cdots,q}V_l>y_n\}\cap\{b_n(\sum_{l=1}^q V_l-a_n)\leq z\}\big)\\
&=&P\big(\cup_{l=1}^q[\{V_l>y_n\}\cap\{b_n(\sum_{l=1}^q V_l-a_n)\leq z\}]\big),\eeaa
where we used the fact that $\{\max_{l=1,\cdots,q}V_l>y_n\}=\{\cup_{l=1}^q[V_l>y_n]\}$ in the second equality. Let $B_l=\{V_l>y_n\}\cap\{b_n(\sum_{l=1}^q V_l-a_n)\leq z\}.$ Then, we have $P_{{S_n,L_n}}(z,y)= P(\cup_{l=1}^q B_l)$.
By using Bonferroni inequality, for any even number $d=O(\log p)<[q/2]$, we know that
\be\lbl{shazi}
 \sum_{s=1}^{d}(-1)^{s-1}\sum_{1\leq l_1<\ldots< l_s\leq q} P(\cap_{t=1}^sB_{l_t}  )
  \leq  P_{{S_n,L_n}}(z,y)
\leq \sum_{s=1}^{d-1}(-1)^{s-1}\sum_{1\leq l_1<\ldots< l_s\leq q} P(\cap_{t=1}^sB_{l_t} )
\ee
and also that
\be\lbl{office}
H_d\leq  P\big(\cup_{l=1}^q\{V_l>y_n\} \big)\leq H_{d-1}
\ee
where $H_d=\sum_{s=1}^{d}(-1)^{s-1}\sum_{1\leq l_1<\ldots< l_s\leq q} P(\cap_{t=1}^s\{V_{l_t}>y_n\})$.

Let $\Upsilon_n=  \min(p^{-1/5}, \sqrt{n/p}) .$ We define two index sets $I=\{(i_{l_t}, j_{l_t}), 1\leq t\leq d) \} $  and $W_{I}=\{(i,j), (i,j)\cap (s,t)\neq \emptyset, \, (s,t)\in I \, \mbox{and}\,   (i,j)\in W \}$. The cardinality of $W_I$ is no greater than ${2pd}$. By construction, $\{V_l, {(i_l, j_l)}\in I\}$ and $\{V_{l'}, {(i_{l'}, j_{l'})}\in W/W_I\}$ are independent. Using the fact that $\sum_{l=1}^q V_l-a_n=\sum_{(i_l, j_l)\in W_{I}} (V_l-1)+\sum_{(i_l, j_l)\in W_{I}} (V_l-1)$, we have

\beaa
   P\big(\cap_{t=1}^dB_{l_t}\big)&\leq&
   P\big(\cap_{t=1}^d\{V_{l_t}>y_n\}\big)P\big( b_n\sum_{(i_l, j_l)\in W \setminus W_{I}}( V_l-1)\leq z+\Upsilon_n\big)\\&&+ P\big(b_n \big|\sum_{(i_l, j_l)\in W_{I}} (V_l-1)\big|\geq \Upsilon_n\big)
\eeaa
 and
\beaa
P\big(\cap_{t=1}^dB_{l_t}\big) &\geq& P\big(\cap_{t=1}^d\{V_{l_t}>y_n\}\big)P\big( b_n\sum_{(i_l, j_l)\in W \setminus W_{I}}( V_l-1)\leq z-\Upsilon_n\big)\\&&- P\big(b_n \big|\sum_{(i_l, j_l)\in W_{I}} (V_l-1)\big|\geq \Upsilon_n\big).
\eeaa
Now, we can follow the proof of Lemma \ref{lawSn} to obtain that
$$|P\big( b_n\sum_{(i_l, j_l)\in W \setminus W_{I}}( V_l-1)\leq z\pm \Upsilon_n\big)-P\big( b_n(\sum_{l=1}^q V_l-a_n)\leq z\big)|\leq C\Upsilon_n .
$$
Combining \eqref{office} and the above inequalities, we have
\beaa
 &&P\big(\cup_{l=1}^qB_{l_t}\big)\\&\leq& H_{d-1} P\big( b_n(\sum_{l=1}^q V_l-a_n)\leq z\big) +CH_{d-1}\Upsilon_n+P\big(b_n \big|\sum_{(i_l, j_l)\in W_{I}} (V_l-1)\big|\geq \Upsilon_n\big)\\&\leq&P\big(\cup_{l=1}^q[\{V_l>y_n\}\big)P\big(\{b_n(\sum_{l=1}^q V_l-a_n)\leq z\}]\big)+|H_d-H_{d-1}|\\&&+C \Upsilon_n+P\big(b_n \big|\sum_{(i_l, j_l)\in W_{I}} (V_l-1)\big|\geq \Upsilon_n\big),\eeaa
where we used the triangle inequality and Bonferroni inequality \eqref{office}   in the second inequality. By Lemma \ref{xia3} and \cite{caixia-2013}, as long as $d\geq 3$,  we have \beaa |H_{d-1}-H_{d}|=\sum_{1\leq l_1<\ldots< l_d\leq q} P(\cap_{s=1}^d \{V_{l_s}\geq y_n\})  \leq(\frac{e^{-y/2+3}}{d})^d,
\eeaa
We choose $d$ be the smallest even number which is bigger than $e^4\theta \log p,$, and then
$$\sup_{-2\log \log p^{\theta}<y<2\log p}(\frac{e^{-y/2+3}}{d})^d\leq  p^{-e^{4}\theta}.$$
We make a claim that
\be\lbl{ken0}
 P\big(b_n \big|\sum_{(i_l, j_l)\in W_{I}} (V_l-1)\big|\geq \Upsilon_n\big)\leq  Ce^{-cn^{1/3}},
\ee
which will be proved later. In view of \eqref{ken0}, we have
$$
P\big(\cup_{l=1}^qB_{l_t}\big)\leq P\big(\cup_{l=1}^q[\{V_l>y_n\}\big)P\big(\{b_n(\sum_{l=1}^q V_l-a_n)\leq z\}]\big)+ C \Upsilon_n.
$$
Similarly, we also have
$$
P\big(\cup_{l=1}^qB_{l_t}\big)\geq P\big(\cup_{l=1}^q[\{V_l>y_n\}\big)P\big(\{b_n(\sum_{l=1}^q V_l-a_n)\leq z\}]\big)-C \Upsilon_n.
$$
Thus, we obtain that
$$
\sup_{-2\log \log p^{\theta}<y<2\log p,z}\Big| P\big(\cup_{l=1}^qB_{l_t}\big)-P\big(\cup_{l=1}^q[\{V_l>y_n\}\big)P\big(\{b_n(\sum_{l=1}^q V_l-a_n)\leq z\}]\big)\Big|\leq C\Upsilon_n.
$$
Combine the above result with (\ref{qusil}), we obtain the desired uniform convergence rate in Theorem 1. Now it only remains to prove the claim (\ref{ken0}). Denote the number of distinct variables with subindex appear in $I$ as L, while ${L\leq 2d}.$  Without loss of generality, We consider the following case. We can bound
    \begin{eqnarray}\lbl{ken}
&& P\big(b_n \sum_{(i_l, j_l)\in W_{I}} (V_l-1)\geq \Upsilon_n\big)\\
&\leq& \sum_{i=1}^{L}P\Big(\sum_{m=i+1}^{p}\Big\{\frac{(\sum_{k=1}^nx_{ki}x_{km})^2}{n}-1\Big\}\geq \frac{ \Upsilon_n}{b_nL}\Big)\nonumber\\
&\leq&  \sum_{i=1}^{L} EP^{i}\Big(\sum_{m=i+1}^{p}\Big\{\frac{(\sum_{k=1}^nx_{ki}x_{km})^2}{n}-\frac{\sum_{k=1}^nx_{ki}^2}{n}\Big\}\geq  \frac{ \Upsilon_n}{2b_nL}\Big)I_{T_n^i} \nonumber \\
&&+CLe^{-n^{1/3}/C},\nonumber
\end{eqnarray}
where $T_n^i=\{|\frac{\sum_{k=1}^nx_{ki}^2}{n}-1|\leq\epsilon n^{-1/3}, |\frac{\sum_{k=1}^nx_{ki}^4}{n}-Ex_{ki}^4|\leq\epsilon n^{-1/3}, \max_{1\leq k\leq n}|x_{ki}|\leq n^{1/6}\}
 $ and $P((T_n^i)^{c})\leq Ce^{-n^{1/3}/C}$. The last inequality is due to that $  \frac{ \Upsilon_n}{b_nL}\gg pn^{-1/3}$.
 Set $h_n=n^{1/3}$ and $\mu_n=E^1 \big(\frac{\sum_{k=1}^nx_{k1}x_{k2}}{\sqrt{n}} \big)^2I(\frac{(\sum_{k=1}^nx_{k1}x_{k2})^2}{n}  \leq h_n).$ 
 Define
    \begin{eqnarray*}
 & &  y_{m}=\big(\frac{\sum_{k=1}^nx_{k1}x_{km}}{\sqrt{n}} \big)^2I(\frac{(\sum_{k=1}^nx_{k1}x_{km})^2}{n}  \leq h_n)  -\mu_n  \nonumber\\
 & & z_{m}=\big(\frac{\sum_{k=1}^nx_{k1}x_{km}}{\sqrt{n}} \big)^2I(\frac{(\sum_{k=1}^nx_{k1}x_{km})^2}{n}> h_n)+\mu_n -\frac{\sum_{k=1}^nx_{k1}^2}{n}\lbl{piano}
 \end{eqnarray*}
 for all $ m\geq 1.$ Use the inequality $P(U+V \geq u+v)\leq P(U\geq u) + P(V\geq v)$ to obtain
 \begin{eqnarray}\lbl{ken1}
 & & P^{1}\Big(\sum_{m=2}^{p}\Big\{\frac{(\sum_{k=1}^nx_{k1}x_{km})^2}{n}-\frac{\sum_{k=1}^nx_{k1}^2}{n}\Big\}\geq  \frac{\Upsilon_n}{2 b_nL}\Big)I_{T_n^1}\nonumber\\
  & \leq & P^{1}\Big(\sum_{m=2}^{p} y_{m}\geq \frac{ \Upsilon_n}{4b_nL}\Big)I_{T_n^1} +
  P^{1}\Big(\sum_{m=2}^{p}z_m\geq  \frac{\Upsilon_n}{4b_nL}\Big)I_{T_n^1}\nonumber\\
 & := & A_n + B_n
 \end{eqnarray}
 for any $n\geq 1$. To bound $A_n$ and $B_n$, we only show the case when $p\gg n^{5/3}$ ,  which indicates $\frac{ \Upsilon_n}{4b_n L}\geq \frac{p}{4L}.$ The other case when $p\le n^{5/3}$  is similar. First we bound $A_n$.
 \begin{eqnarray}\lbl{ken2}
 A_n
 &\leq & 4\cdot\exp\Big\{-\frac{p^2/32C^2L^2}{ \frac{(p-1)}{n^2}(\sum\limits_{k=1}^nx_{k1}^4\min\limits_{1 < i\leq p}Ex_{ki}^4+\sum\limits_{1\leq k\neq l\leq n}x_{k1}^2x_{l1}^2) +  \frac{h_np}{12L}} \Big\}I_{T_n^1}u\nonumber\\
 & \leq & 4\cdot\exp\Big\{-\frac{p }{3Ch_nL}\Big\}.
 \end{eqnarray}
Define $b_2=x_{k2}I(|x_{k2}|\leq n^{1/6})-Ex_{k2}I(|x_{k2}|\leq n^{1/6})$ and $b_3=x_{k2}I(|x_{k2}|>n^{1/6})-Ex_{k2}I(|x_{k2}|> n^{1/6})$.
 By Bernstein inequality, the bound of $B_n$,
  \begin{eqnarray}\lbl{kenp}
 &&  P^1\big (\frac{(\sum_{k=1}^nx_{k1}x_{k2})^2}{n}  \geq h_n\big)I_{T_n^1}\nonumber \\&\leq& 2 P^1\big(\sum_{k=1}^n x_{k1}b_2  \geq n^{2/3}/2\big)I_{T_n^1}+2P^1\big(\sum_{k=1}^n x_{k1}b_3  \geq n^{2/3}/2\big)I_{T_n^1}\nonumber \\&\leq& 2\exp\big\{-\frac{n^{4/3}}{8 \sum_{k=1}^{n}x_{k1}^2+8/3 \max{|x_{k1}|}n^{5/6}}\big\}I_{T_n^1}\nonumber\\&& +
  2P(\max_{1\leq k\leq n}|x_{k2}|>n^{1/6})\leq Ce^{-n^{1/3}/C}.
   \end{eqnarray}
    Since   \begin{eqnarray*}
    &&|\mu_n-\frac{\sum_{k=1}^nx_{k1}^2 }{n}|I_{T_n^1}\\
    &\leq &C\frac{\sum_{k=1}^nx_{k1}^2+\sum_{k_1\neq k_2}x_{k_11}x_{k_21}}{n} I_{T_n^1}[P^1\big (\frac{(\sum_{k=1}^nx_{k1}x_{k2})^2}{n}\geq h_n\big) I_{T_n^1}]^{1/2}\\&\leq& Ce^{-n^{1/3}/C}, \end{eqnarray*} then $B_n\leq pP^1\big (\frac{(\sum_{k=1}^nx_{k1}x_{k2})^2}{n}  \geq h_n\big)I_{T_n^1}.$
   Combine (\ref{ken}), (\ref{ken2}) and (\ref{kenp}), we show the claim.

Therefore, The proof of Theorem 1 is complete.
\end{proof}

\subsection{Proof of Theorem 2}

\begin{proof}[Proof of Theorem 2]
This is a direct result by combining Theorem \ref{birthday} and Lemmas \ref{lawSn}--\ref{lawLn}.
\end{proof}

\subsection{Proof of Theorem 3}

The follow  Lemma \ref{xue2} facilitates the proof of Theorem 3, and the  proof of Lemma \ref{xue2} is given in \cite{lixue-2015}.
\begin{lemma}\lbl{xue2}   Assume $n,p\to \infty$, $d=O(\log p) $ and $\log p=o(n^{1/6})$ as $n\to\infty.$ We arrange the tow dimensional indices $\{(i,j): 1\leq i<j\leq p\}$ in any ordering and set them as $\{(i_l, j_l): 1\leq l\leq q={p \choose 2}\}.$ Let $\delta_l=(r_{i_lj_l})^2/n $.    \beaa
 \sum_{1\leq l_1<\ldots< l_d\leq q}  P(\cap_{t=1}^d \{\delta_{l_t}\geq 4\log p - \log\log p + y\})\leq(\frac{e^{-y/2+3}}{d})^d.\eeaa
 \end{lemma}

Now we are ready to prove Theorem 3.

\begin{proof}[Proof of Theorem \ref{birthdaycake}] ${The}$ proof is similar to  Theorem \ref{birthday}. We only need to replace Lemma \ref{xia3} with Lemma \ref{xue2} and also prove the following claim
 \begin{eqnarray}\lbl{aircrash}
 P\big(\beta_n \big|\sum_{(i_l, j_l)\in W_{I}} (\delta_l-1-\frac{1}{n-1})\big|\geq \Gamma_n\big)&\leq&    Ce^{-n^{1/3}/C}
\end{eqnarray}
where $\delta_l=nr_{ij}^2$, $ \Gamma_n=\min(n^2,p)^{-1/5}$ and $W_I$ is defined as in the proof of Theorem \ref{birthday}.
Thus, it remains to prove the claim (\ref{aircrash}).
 Denote the number of  distinct vectors with subindex in $I$ as L. Then,
    \begin{eqnarray}\lbl{qusi}
&& P\big(\beta_n \sum_{(i_l, j_l)\in W_{I}} (\delta_l-1-\frac{1}{n-1})\geq \Gamma_n\big)\nonumber\\ &\leq& \sum_{i=1}^{L}P\Big(\sum_{m=i+1}^{p}\Big\{\frac{(\sum_{k=1}^nN_{ki}N_{km})^2}{n}-1-\frac{1}{n-1}\Big\}\geq \frac{\Gamma_n}{\beta_nL}\Big)\nonumber
 \\&=&  \sum_{i=1}^{L} EP^{i}\Big(\sum_{m=i+1}^{p}\Big\{\frac{(N_i'N_m)^2}{n}- 1-\frac{1}{n-1} \Big\}\geq  \frac{\Gamma_n}{ \beta_nL}\Big).
\end{eqnarray}

  Set $h_n=n^{1/3}$. 
 Define
    \begin{eqnarray*}
    &&\mu_n=E^1 \big(\frac{\sum_{k=1}^nN_{k1}N_{k2}}{\sqrt{n}} \big)^2I(\frac{(\sum_{k=1}^nN_{k1}N_{k2})^2}{n}  \leq h_n),\nonumber\\
 & &  y_{m}=\big(\frac{\sum_{k=1}^nN_{k1}N_{km}}{\sqrt{n}} \big)^2I(\frac{(\sum_{k=1}^nN_{k1}N_{km})^2}{n}  \leq h_n)  -\mu_n,  \nonumber\\
 & & z_{m}=\big(\frac{\sum_{k=1}^nN_{k1}N_{km}}{\sqrt{n}} \big)^2I(\frac{(\sum_{k=1}^nN_{k1}N_{km})^2}{n}> h_n)+\mu_n - 1-\frac{1}{n-1}
 \end{eqnarray*}
 for all $ m\geq 1.$ Use the inequality $P(U+V \geq u+v)\leq P(U\geq u) + P(V\geq v)$ to obtain
 \begin{eqnarray}
 & & P^{1}\Big(\sum_{m=2}^{p}\Big\{\frac{(\sum_{k=1}^nN_{k1}N_{km})^2}{n}-\frac{N_1'RN_1}{n}\Big\}\geq  \frac{\Gamma_n}{\beta_nL}\Big) \nonumber\\
  & \leq & P^{1}\Big(\sum_{m=2}^{p} y_{m}\geq \frac{\Gamma_n}{2\beta_nL}\Big)  +
  P^{1}\Big(\sum_{m=2}^{p}z_m\geq  \frac{\Gamma_n}{2\beta_nL}\Big) \nonumber\\
 & := & A_n + B_n
 \end{eqnarray}
 for any $n\geq 1$. First we bound $A_n$.
     Since   \begin{eqnarray*} &&E^1y_m^2\\&\leq& E^1\big(\frac{\sum_{k=1}^nN_{k1}N_{km}}{\sqrt{n}} \big)^4 \\
    &=&\frac{1}{n^2}\big[EN_{11}^4
   \sum_{i=1}^n N_{i1}^4 +\sum_{i\neq j}^{n}\{3(EN_{11}^2N_{21}^2)N_{i1}^2N_{j_1}^2 +3(EN_{11}^3N_{21})N_{i1}^3N_{j1}\}\\
 &&+\sum_{i\neq j\neq s}\{6(EN_{11}^2N_{21}N_{31})N_{i1}^2N_{j1}N_{s1} \} +
\sum_{i\neq j\neq s\neq t}(EN_{11}N_{21}N_{31}N_{41})N_{i1}N_{j1}N_{s1}N_{t1}\big]\\&\leq C_0&
     \end{eqnarray*} then by Bernstein inequality,
 \begin{eqnarray*}
 A_n &\leq & 4\cdot\exp\Big\{-\frac{\Gamma_n^2/8\beta_n^2L^2}{ (p-1) C_0 +  h_np/12L} \Big\}
  \leq  4\cdot\exp\Big\{-\frac{n^{1/3} }{C}\Big\}.
 \end{eqnarray*}
    By the large deviation for generalized rank statistics,   the bound of $B_n$,
  \begin{eqnarray}\lbl{ken3}
 && EB_n\leq EP^1\big (\frac{(\sum_{k=1}^nN_{1k}N_{2k})^2}{n}  \geq h_n\big) \leq Ce^{-n^{1/3}/C}.
   \end{eqnarray}

Therefore, the proof of Theorem 3 is complete.
   \end{proof}

\subsection{Proof of Theorem 4}

\begin{proof}[Proof of Theorem 4]
This is a direct result by combining Theorem \ref{birthdaycake} and Lemmas \ref{lawTn}--\ref{lawMn}.
\end{proof}

\subsection{Proof of Theorem 5}

\begin{proof}[Proof of Theorem 5]
This is a direct result by using Theorem \ref{prasing} and Theorem \ref{rasing}.
\end{proof}

\subsection{Proof of Theorem 6}

\begin{proof} [Proof of Theorem \ref{tb1}]
Proof is divided into two parts. In the first part, we show that $\frac{1}{B_n}(S_n-A_n)\to N(0,1)$,
where $A_n=\frac{n^2+1}{2n^2}[\text{tr}(\Sigma^2)-p]+\frac{p(p-1)}{2n}$ and
\begin{eqnarray*}
  B_n^2&=&\frac{1}{n^3}(2(n-1)^2\text{tr}(\Sigma^4)+4(n-1)(p-n)\text{tr}(\Sigma^3)+2(p-n)^2\text{tr}(\Sigma^2))\\
  &&+\frac{n-1}{n^3}(\text{tr}(\Sigma^2)^2+\text{tr}(\Sigma^4)-4\sum_{i,j,s}\sigma_{ij}^2\sigma_{is}^2+4\sum_{i,j}\sigma_{ij}^4). \end{eqnarray*}
To this end, we decompose $S_n-A_n$ as follows:
$$
 S_n-A_n=\sum_{m=2}^n\sum_{l=1}^{m-1}H_n( \mathbf{x}_{ m}, \mathbf{x}_{ l})+\sum_{m=1}^n (Y_m+U_m)+R_1,
$$
where $H_n( \mathbf{x} _{ m},  \mathbf{x}_{ l})=\frac{1}{n^2}\sum_{i\neq j}(x_{mi}x_{mj}-\sigma_{ij})(x_{li}x_{lj}-\sigma_{ij})$, $Y_m=\frac{n-1}{n^2}\sum_{i\neq j}\sigma_{ij}(x_{mi}x_{mj}-\sigma_{ij})$, $U_m=\frac{1}{n^2}\sum_{i\neq j}\sigma_{ii}(x_{mj}^2-\sigma_{jj})$, and $R_1=\frac{1}{n^2}\sum_{m=1}^n\sum_{i<j}\{(x_{mi}^2-\sigma_{ii})(x_{mj}^2-\sigma_{jj})-2\sigma_{ij}^2\}$. We define $G_1=Y_1+U_1$ and $G_m=\sum_{l=1}^{m-1}H_n(\mathbf{x}_{ m}, \mathbf{x}_{ l})+Y_m+U_m$ for $m\geq 2$. Let $W_k=\sum_{m=1}^kG_m$  and $\mathcal{F}_{k}$ denote the $\sigma$-field generated by $( \mathbf{x}_1, \ldots,  \mathbf{x}_k)$ for $k\geq 1$. Then $\{W_{m}, \mathcal{F}_m, m=1, 2, \ldots \}$ is a martingale. Now, we have $S_n-A_n=W_n+R_1$, where $R_1$ is the higher order residual. It is not difficult to show that $ EW_n^2=B_n^2$. Thus, the desired result can be obtained by applying the martingale central limit theorem \citep{haeusler-1988} to $W_n$. To this end, we only need to show that
 \begin{eqnarray}\lbl{cat}
B_{n}^{-2}\sum_{m=1}^{n} EG_m^2I(|G_m|>\epsilon B_{n})\to 0
    \end{eqnarray}
    as $n\to \infty$ for each $\epsilon>0,$ and also that
       \begin{eqnarray}\lbl{dog}
       B_{n}^{-2} \sum_{m=1}^nE_{m-1}G_{m}^2\to 1\,\,\,\,\text{ in probability}.
        \end{eqnarray}
We first prove that \eqref{cat} is satisfied. Note that  \begin{eqnarray*}          &&\sum_{m=1}^nEG_{m}^4\\&\leq& 27\sum_{m=1}^n \{\sum_{l=1}^{m-1}E H_n^4( \mathbf{x}_{ m}, \mathbf{x}_{ l})+3\sum_{l_1\neq l_2}H_n^2(\mathbf{x}_{ m}, \mathbf{x}_{ l_1})H_n^2( \mathbf{x}_{ m}, \mathbf{x}_{ l_2})+EY_m^4+EU_m^4\}\\&\leq& 27(n^3+n^2)EH_n^4( \mathbf{x}_{ 1},  \mathbf{x}_{ 2})+27nEY_1^4+27nEU_1^4.
        \end{eqnarray*}
By using Lemma 10 of \cite{lizou-2014}, we already know that $EH_n^4(\mathbf{x}_{ 1}, \mathbf{x}_{ 2})\leq \frac{C}{n^8}\text{tr}(\Sigma^2)^4$, $EY_1^4\leq \frac{C}{n^4}\{\text{tr}(\Sigma^4)^2+\text{tr}(\Sigma^8)+\text{tr}(\Sigma^6)+\text{tr}(\Sigma^2)\text{tr}(\Sigma^4)+\text{tr}^2(\Sigma^2)+\text{tr}(\Sigma^4)\}\leq  \frac{2C}{n^4}\text{tr}(\Sigma^4)^2
       $ and $EU_1^4\leq\frac{C}{n^8}p^4 \text{tr}(\Sigma^2)^2.$
Combining these results, we obtain (\ref{cat}).

Next we prove \eqref{dog}. We define $D_m=E_{m-1}G_m^2:=D_{m1}+D_{m2}+D_{m3}+D_{m4}$, where
        \begin{eqnarray*}
        D_{m1}&=&\frac{1}{n^4}\sum_{l=1}^{m-1}\sum_{i\neq j, s\neq t}(\sigma_{is}\sigma_{jt}+\sigma_{it}\sigma_{js})(x_{li}x_{lj}-\sigma_{ij})(x_{fs}x_{ft}-\sigma_{st}),\\
          D_{m2}&=&\frac{1}{n^4}\sum_{l\neq f}^{m-1}\sum_{i\neq j, s\neq t}(\sigma_{is}\sigma_{jt}+\sigma_{it}\sigma_{js})(x_{li}x_{lj}-\sigma_{ij})(x_{fs}x_{ft}-\sigma_{st}),\\
                  D_{m3}&=&\frac{n-1}{n^4}\sum_{l=1}^{m-1}\sum_{i\neq j, s\neq t}\sigma_{s t} (\sigma_{is}\sigma_{jt}+\sigma_{it}\sigma_{js})(x_{li}x_{lj}-\sigma_{ij}),\\
                  D_{m4}&=&\frac{1}{n^4}\sum_{l=1}^{m-1}\sum_{i\neq j, s\neq t} (\sigma_{ss}\sigma_{it}\sigma_{jt}+\sigma_{tt}\sigma_{is}\sigma_{js})(x_{li}x_{lj}-\sigma_{ij}).
                        \end{eqnarray*}
By simple calculation, it is easy to see that $\text{Var}(\sum_{m=1}^n D_m)= \sum_{m=1}^n(2n-2m+1)\text{Var}(D_m).$
Then we prove (\ref{dog}) by noting that
\beaa
&&\text{Var}(\sum_{m=1}^n D_m)\\&\leq& Cn^3\text{Var}(D_{m1})+Cn^4\text{Var}(D_{m2})+Cn^2\text{Var}(D_{m3})+Cn^2\text{Var}(D_{m4})\\
&\leq& C \frac{\text{tr}^4(\Sigma^2)}{n^5}+C\frac{\text{tr}^2(\Sigma^4)}{n^5}
+C\frac{\text{tr}(\Sigma^8)}{n^3}+C\frac{\text{tr}(\Sigma^8)}{n^5}+C\frac{p^2\text{tr}(\Sigma^6)+p^2\text{tr}^2(\Sigma^2)}{n^5}.
\eeaa

In the second part, we prove the power consistency of $TS_n^1$ against $\mathbf{H}_1:\Sigma\in \mathcal{G}_1\cup \mathcal{G}_2 $, i.e.,  $\inf_{\Sigma\in \mathcal{G}_1\cup \mathcal{G}_2 }P (TS_n^1=1)\to 1, \ \text{as} \ n\rightarrow \infty.$  Recall that the threshold $c_\alpha$ is the $\alpha$ upper quantile of $\Phi\star F$. To simplify notation, we let $SL_n=b_n(S_n-a_n)+(n L_{n}^2 - 4\log p + \log\log p)\geq c_\alpha$. It is obvious that
\beaa
\inf_{\Sigma\in \mathcal{G}_1\cup \mathcal{G}_2 }P (TS_n^1=1)&=&\inf_{\Sigma\in \mathcal{G}_1\cup \mathcal{G}_2 }P (SL_n\geq c_\alpha)\\
&\ge&\min(\inf _{\Sigma\in  \mathcal{G}_1 } P (SL_n>c_{\alpha}),\inf _{\Sigma\in  \mathcal{G}_2 }P (SL_n>c_{\alpha})).
\eeaa
On the one hand, we have the simple probability bound that
\beaa \inf_{\Sigma\in  \mathcal{G}_1 } P (SL_n\ge c_{\alpha}) &\geq&   \inf_{\Sigma\in  \mathcal{G}_1 }P(nL_n-4\log p+\log\log p\geq \frac{1}{2}\log p+c_{\alpha})\\&&-  \sup_{\Sigma\in  \mathcal{G}_1 }P( b_n(nS_n-a_n)\leq- \frac{1}{2}\log p). \eeaa
In what follows, we shall show that $\inf_{\Sigma\in  \mathcal{G}_1 }P (nL_n-4\log p+\log\log p\geq \frac{1}{2}\log p+c_{\alpha})\rightarrow 1$ and $\inf_{\Sigma\in  \mathcal{G}_1 } P ( b_n(nS_n-a_n)\leq- \frac{1}{2}\log p)\rightarrow 0$ as $n$ diverges to infinity. It is not difficult to show that $n^2b_n^2=O(\frac{n^3}{np^2+p^3})$  and $A_n-a_n/n=\frac{n^2+1}{2n^2}tr(\Sigma-I)^2.$ When $n $ is large enough, we can further show that $B_n^2\leq \frac{4(n-1)^2tr(\Sigma^2)tr(\Sigma-I)^2+4(p-1)^2tr(\Sigma^2)+4(n-1)[tr(\Sigma^2)]^2}{n^3}$ and also that $B_n^2\geq \frac{2(n-1)^2tr(\Sigma^2-\Sigma)^2+2(p-1)^2tr(\Sigma^2)+ (n-1)[tr(\Sigma^2)]^2}{n^3}$. Then, we have
 \begin{eqnarray*} &&P( b_n(nS_n-a_n)\leq- \frac{1}{2}\log p)\\&\leq&
 P(\frac{1}{B_n}(S_n-A_n)\leq- \frac{\log p}{2nb_nB_n}-\frac{A_n-a_n/n}{B_n} )
\\ &\to& \Phi(- \frac{\log p}{2nb_nB_n}-\frac{A_n-a_n/n}{B_n}). \end{eqnarray*}
Hence, $\sup_{\Sigma\in  \mathcal{G}_1 }P ( b_n(nS_n-a_n)\leq- \frac{1}{2}\log p)\rightarrow 0$ holds since $ \frac{\log p}{2nb_nB_n}+\frac{A_n-a_n/n}{B_n}\to \infty$ as $n\to \infty$, for all ${\Sigma\in  \mathcal{G}_1 }$. In the meantime, we also have
 \begin{eqnarray*}
   &&\inf_{\Sigma\in  \mathcal{G}_1 }P (nL_n-4\log p+\log\log p\geq \frac{1}{2}\log p+c_{\alpha})\\ && \geq   \inf_{\Sigma\in  \mathcal{G}_1 }P(\max|\sigma_{ij}|-\max_{ij}|\hat{\sigma}_{ij}-\sigma_{ij}|\geq \frac{9}{2}\log p-\log\log p)\\&& \geq 1-\sup_{\Sigma\in  \mathcal{G}_1 }p(\max_{ij}|\hat{\sigma}_{ij}-\sigma_{ij}|\geq( C-\frac{9}{2})\sqrt{\log p/n}).
 \end{eqnarray*}
Thus,  $\inf_{\Sigma\in  \mathcal{G}_1 }P (nL_n-4\log p+\log\log p\geq \frac{1}{2}\log p+c_{\alpha})\to 1$, when $p=O(n^3)$ and  $n\to \infty.$ We immediately obtain that $\inf_{\Sigma\in  \mathcal{G}_1 } P (SL_n\ge c_{\alpha})\to 1$.

On the other hand, we use the simple probability bound again to obtain that
\beaa
\inf_{\Sigma\in  \mathcal{G}_2 }P(SL_n>c_{\alpha}) &\geq&     \inf_{\Sigma\in  \mathcal{G}_2 }P( b_n(nS_n-a_n)\geq 4\log p+c_{\alpha})\\&&-\sup_{\Sigma\in  \mathcal{G}_2 }P(nL_n-4\log p+\log\log p\leq -4\log p).
\eeaa
It is obvious that $P(nL_n-4\log p+\log\log p\leq -4\log p)=0$. Moreover, as $n\to \infty$,
$$
  \inf_{\Sigma\in  \mathcal{G}_2 }P( b_n(nS_n-a_n)\geq 4\log p+c_{\alpha})\to1-\Phi(4 \frac{\log p}{2nb_nB_n}-\frac{A_n-a_n/n}{B_n})    \to 1,
$$
Thus, we obtain that $\inf_{\Sigma\in  \mathcal{G}_2 }P(SL_n\ge c_{\alpha})\to 1$.

Therefore, the proof of Theorem 6 is complete.
\end{proof}

\end{document}